\documentclass[reqno,a4paper, 11pt]{amsart}

\usepackage{pictexwd,dcpic,amssymb}
\usepackage{graphicx}
\usepackage{array}
\usepackage{multirow}

\newtheorem{theorem}{Theorem}[section]
\newtheorem{proposition}[theorem]{Proposition}
\newtheorem{lemma}[theorem]{Lemma}

\theoremstyle{definition}

\newtheorem{example}[theorem]{Example}

\theoremstyle{remark}
\newtheorem{remark}[theorem]{Remark}

\numberwithin{equation}{section}

\begin{document}

\title{Normal bases of ray class fields over imaginary quadratic fields}

\author{Ho Yun Jung}
\address{Department of Mathematical Sciences, KAIST}
\curraddr{Daejeon 373-1, Korea} \email{DOSAL@kaist.ac.kr}
\thanks{}

\author{Ja Kyung Koo}
\address{Department of Mathematical Sciences, KAIST}
\curraddr{Daejeon 373-1, Korea} \email{jkkoo@math.kaist.ac.kr}
\thanks{}

\author{Dong Hwa Shin}
\address{Department of Mathematical Sciences, KAIST}
\curraddr{Daejeon 373-1, Korea} \email{shakur01@kaist.ac.kr}
\thanks{}

\subjclass[2010]{11F11, 11F20, 11R37, 11Y40}

\keywords{class fields, modular functions, normal bases.
\newline This research was partially supported by Basic Science Research
Program through the NRF of Korea funded by MEST (2010-0001654). The
third named author is partially supported by TJ Park Postdoctoral
Fellowship.}

\maketitle

\begin{abstract}
We develop a criterion for a normal basis (Theorem \ref{criterion}),
and prove that the singular values of certain Siegel functions form
normal bases of ray class fields over imaginary quadratic fields
other than $\mathbb{Q}(\sqrt{-1})$ and $\mathbb{Q}(\sqrt{-3})$
(Theorem \ref{main}). This result would be an answer for the
Lang-Schertz conjecture on a ray class field with modulus generated
by an integer ($\geq2$) (Remark \ref{S-Rinvariant}).
\end{abstract}

\section{Introduction}

Let $L$ be a finite Galois extension of a field $K$. From the normal
basis theorem (\cite{Waerden}) we know that there exists a normal
basis of $L$ over $K$, namely, a basis of the form
$\{x^\gamma:\gamma\in\mathrm{Gal}(L/K)\}$ for a single element $x\in
L$.
\par
Okada (\cite{Okada1}) showed that if $k$ and $q$ ($>2$) are positive
integers with $k$ odd and $T$ is a set of representatives for which
$(\mathbb{Z}/q\mathbb{Z})^\times=T\cup(-T)$, then the real numbers
$(\frac{1}{\pi}\frac{d}{dz})^k(\cot \pi z)|_{z=a/q}$ for $a\in T$
form a normal basis of the maximal real subfield of
$\mathbb{Q}(e^{2\pi i/q})$ over $\mathbb{Q}$. Replacing the
cotangent function by the Weierstrass $\wp$-function with
fundamental period $i$ and $1$, he further obtained in \cite{Okada2}
normal bases of class fields over Gauss' number field
$\mathbb{Q}(\sqrt{-1})$. This result was due to the fact that Gauss'
number field has class number one, which can be naturally extended
to any imaginary quadratic field of class number one.
\par
After Okada, Taylor (\cite{Taylor}) and Schertz (\cite{Schertz2})
established Galois module structures of rings of integers of certain
abelian extensions over an imaginary quadratic field, which are
analogues to the cyclotomic case (\cite{Leopoldt}). They also found
normal bases by making use of special values of modular functions.
And, Komatsu (\cite{Komatsu}) considered certain abelian extensions
$L$ and $K$ of $\mathbb{Q}(e^{2\pi i/5})$ and constructed a normal
basis of $L$ over $K$ in terms of special values of Siegel modular
functions.
\par
For any pair $(r_1,r_2)\in\mathbb{Q}^2-\mathbb{Z}^2$ we define the
\textit{Siegel function} $g_{(r_1,r_2)}(\tau)$ on $\mathfrak{H}$
(the complex upper half-plane) by the following infinite product
\begin{eqnarray*}
g_{(r_1,r_2)}(\tau)=-q_\tau^{(1/2)\textbf{B}_2(r_1)}e^{\pi
ir_2(r_1-1)}(1-q_z)\prod_{n=1}^{\infty}(1-q_\tau^nq_z)(1-q_\tau^nq_z^{-1}),
\end{eqnarray*}
where $\textbf{B}_2(X)=X^2-X+1/6$ is the second Bernoulli
polynomial, $q_\tau=e^{2\pi i\tau}$ and $q_z=e^{2\pi iz}$ with
$z=r_1\tau+r_2$. It is a modular unit in the sense of \cite{K-L}.
\par
Let $K$ ($\neq\mathbb{Q}(\sqrt{-1}),~\mathbb{Q}(\sqrt{-3})$) be an
imaginary quadratic field of discriminant $d_K$ and
$\mathcal{O}_K=\mathbb{Z}[\theta]$ be its ring of integers with
\begin{eqnarray}\label{theta}
\theta=\left\{\begin{array}{ll}\sqrt{d_K}/2&\textrm{for}~d_K\equiv0\pmod{4}\vspace{0.1cm}\\
(-1+\sqrt{d_K})/2&\textrm{for}~
d_K\equiv1\pmod{4}.\end{array}\right.
\end{eqnarray}
In what follows we denote the Hilbert class field and the ray class
field modulo $N\mathcal{O}_K$ for an integer $N$ ($\geq2$) by $H$
and $K_{(N)}$, respectively. We showed in \cite{J-K-S} that the
singular value $x=g_{(0,1/N)}(\theta)^{-12N/\gcd(6,N)}$ is a
primitive generator of $K_{(N)}$ over $K$. We achieved this result
by showing that the absolute value of $x$ is the smallest one among
those of all the conjugates.
\par
In this paper we will show that the conjugates of a high power of
$x$ form a normal basis of $K_{(N)}$ over $K$ by applying a
criterion for a normal basis (Theorems \ref{criterion} and
\ref{main}). As for the action of $\mathrm{Gal}(K_{(N)}/K)$ on $x$
in the process we adopt the idea of Gee-Stevenhagen (\cite{Gee},
\cite{Stevenhagen}). Our result is also related to the Lang-Schertz
conjecture on the Siegel-Ramachandra invariant to construct ray
class fields $K_{(N)}$ (Remark \ref{S-Rinvariant}).

\section{A criterion for a normal basis}

In this section we let $L$ be a finite abelian extension of a number
field $K$ with
$G=\mathrm{Gal}(L/K)=\{\gamma_1=\mathrm{id},\cdots,\gamma_n\}$.
Furthermore, we denote by $|\cdot|$ the usual absolute value defined
on $\mathbb{C}$.

\begin{lemma}\label{det}
A set of elements $\{x_1,\cdots,x_n\}$ in $L$ is a $K$-basis of $L$
if and only if
\begin{equation*}
\det(x_i^{\gamma_j^{-1}})_{1\leq i,j\leq n}\neq0.
\end{equation*}
\end{lemma}
\begin{proof}
Straightforward.
\end{proof}

By $\widehat{G}$ we mean the character group of $G$.

\begin{lemma}[Frobenius determinant relation]\label{Frobenius}
If $f$ is any $\mathbb{C}$-valued function on $G$, then
\begin{equation*}
\prod_{\chi\in\widehat{G}} \sum_{1\leq i\leq
n}{\chi}(\gamma_i^{-1})f(\gamma_i)=\det(f(\gamma_i
\gamma_j^{-1}))_{1\leq i,j\leq n}.
\end{equation*}
\end{lemma}
\begin{proof}
See \cite{Lang} Chapter 21 Theorem 5.
\end{proof}

Combining Lemmas \ref{det} and \ref{Frobenius} we derive the
following proposition.

\begin{proposition}\label{character}
The conjugates of an element $x\in L$ form a normal basis of $L$
over $K$ if and only if
\begin{equation*}
\sum_{1\leq i\leq n }{\chi}(\gamma_i^{-1})x^{\gamma_i}\neq0
\quad\textrm{for all}~\chi\in\widehat{G}.
\end{equation*}
\end{proposition}
\begin{proof}
For an element $x\in L$, set $x_i=x^{\gamma_i}$ ($i=1,\cdots,n$). We
establish that
\begin{eqnarray*}
&&\textrm{the conjugates of $x$ form a normal basis of $L$ over
$K$}\\
&\Longleftrightarrow&\textrm{\{$x_1,\cdots,x_n\}$ is a $K$-basis of
$L$}~\textrm{by the definition of a normal basis}\\
&\Longleftrightarrow& \det(x_i^{\gamma_j^{-1}})_{1\leq i,j\leq
n}\neq0\quad\textrm{by Lemma \ref{det}}\\
&\Longleftrightarrow& \sum_{1\leq i\leq
n}{\chi}(\gamma_i^{-1})x_i\neq0\quad\textrm{for all
$\chi\in\widehat{G}$ by Lemma \ref{Frobenius} with}~f(\gamma_i)=x_i.
\end{eqnarray*}
\end{proof}

Now we present a useful criterion which enables us to determine
whether the conjugates of an element $x\in L$ form a normal basis of
$L$ over $K$.

\begin{theorem}\label{criterion}
Assume that there exists an element $x\in L$ such that
\begin{equation}\label{smaller}
|x^{\gamma_i}/x|<1\quad\textrm{for $1<i\leq n$}.
\end{equation} Then the conjugates of a high power
of $x$ form a normal basis of $L$ over $K$.
\end{theorem}
\begin{proof}
By the hypothesis (\ref{smaller}) we can take a suitably large
integer $m$ such that
\begin{equation}\label{smaller2}
|x^{\gamma_i}/x|^m\leq 1/\#G\quad\textrm{for $1<i\leq n$},
\end{equation}
where $\#G$ is the cardinality of $G$. Then, for any
$\chi\in\widehat{G}$ we derive that
\begin{eqnarray*}
|\sum_{1\leq i\leq n}{\chi}(\gamma_i^{-1})(x^m)^{\gamma_i}| &\geq&
|x^m|(1-\sum_{1<i\leq n}|(x^m)^{\gamma_i}/x^m|)\quad\textrm{by the
triangle inequality}\\
&\geq&|x^m|(1-(1/\#G)(\#G-1))=|x^m|/\#G>0\quad\textrm{by
(\ref{smaller2}).}
\end{eqnarray*}
Therefore the conjugates of $x^m$ form a normal basis of $L$ over
$K$ by Proposition \ref{character}.
\end{proof}

\section{Actions of Galois groups}\label{section3}

We shall investigate an algorithm to get all conjugates of the
singular value of a modular function.
\par
For each positive integer $N$, let $\mathcal{F}_N$ be the field of
modular functions of level $N$ defined over $\mathbb{Q}(\zeta_N)$
with $\zeta_N=e^{2\pi i/N}$. Then, $\mathcal{F}_N$ is a Galois
extension of $\mathcal{F}_1=\mathbb{Q}(j)$ ($j=$the elliptic modular
function) whose Galois group is isomorphic to
$\mathrm{GL}_2(\mathbb{Z}/N\mathbb{Z})/\{\pm1_2\}$ (\cite{Lang} or
\cite{Shimura}).
\par
Throughout this section we let $K$ be an imaginary quadratic field
of discriminant $d_K$ and $\theta$ be as in (\ref{theta}).
\par
Under the properly equivalence relation, primitive positive definite
binary quadratic forms $aX^2+bXY+cY^2$ of discriminant $d_K$
determine a group $\mathrm{C}(d_K)$, called the \textit{form class
group of discriminant $d_K$}. We identify $\mathrm{C}(d_K)$ with the
set of all \textit{reduced quadratic forms}, which are characterized
by the conditions
\begin{equation}\label{reduced}
-a<b\leq a<c\quad\textrm{or}\quad 0\leq b\leq a=c
\end{equation}
together with the discriminant relation
\begin{equation}\label{disc}
b^2-4ac=d_K.
\end{equation}
It is well-known that $\mathrm{C}(d_K)$ is isomorphic to
$\mathrm{Gal}(H/K)$ (\cite{Cox}). For a reduced quadratic form
$Q=aX^2+bXY+cY^2$ in $\mathrm{C}(d_K)$ we define a CM-point
\begin{equation}\label{theta_Q}
\theta_Q=(-b+\sqrt{d_K})/2a.
\end{equation}
Furthermore, we define
$\beta_Q=(\beta_p)_p\in\prod_{p~:~\textrm{primes}}\mathrm{GL}_2(\mathbb{Z}_p)$
as
\begin{itemize}
\item[] Case 1 : $d_K\equiv0\pmod{4}$
\begin{eqnarray}\label{u1}
\beta_p=\left\{\begin{array}{ll}
\begin{pmatrix}a&b/2\\0&1\end{pmatrix}&\textrm{if}~p\nmid a\vspace{0.1cm}\\
\begin{pmatrix}-b/2&-c\\1&0\end{pmatrix}&\textrm{if}~p\mid a~\textrm{and}~p\nmid c\vspace{0.1cm}\\
\begin{pmatrix}-a-b/2&-c-b/2\\1&-1\end{pmatrix}&\textrm{if}~p\mid a~\textrm{and}~p\mid c
\end{array}\right.
\end{eqnarray}
\item[] Case 2 : $d_K\equiv1\pmod{4}$
\begin{eqnarray}\label{u2}
\beta_p=\left\{\begin{array}{ll}
\begin{pmatrix}a&(b-1)/2\\0&1\end{pmatrix}&\textrm{if}~p\nmid a\vspace{0.1cm}\\
\begin{pmatrix}-(b+1)/2&-c\\1&0\end{pmatrix}&\textrm{if}~p\mid a~\textrm{and}~p\nmid c\vspace{0.1cm}\\
\begin{pmatrix}-a-(b+1)/2&-c-(b-1)/2\\1&-1\end{pmatrix}&\textrm{if}~p\mid a~\textrm{and}~p\mid
c.
\end{array}\right.
\end{eqnarray}
\end{itemize}

The next two propositions describe Shimura's reciprocity law
explicitly (\cite{Lang} or \cite{Shimura}).

\begin{proposition}\label{Gal(H/K)}
Let $N$ be a positive integer. If $h\in\mathcal{F}_N$ is defined and
finite at $\theta$ and $Q$ is a reduced quadratic form in
$\mathrm{C}(d_K)$, then the value $h^{\beta_Q}(\theta_Q)$ belongs to
$K_{(N)}$. Here, we note that there exists
$\beta\in\mathrm{GL}_2^+(\mathbb{Q})\cap\mathrm{M}_2(\mathbb{Z})$
such that $\beta\equiv \beta_p\pmod{N\mathbb{Z}_p}$ for all primes
$p$ dividing $N$ by the Chinese remainder theorem. Then the action
of $\beta_Q$ on $\mathcal{F}_N$ is understood as that of $\beta$
which is an element of
$\mathrm{GL}_2(\mathbb{Z}/N\mathbb{Z})/\{\pm1_2\} $
($\simeq\mathrm{Gal}(\mathcal{F}_N/\mathcal{F}_1)$). Furthermore, we
have an isomorphism
\begin{eqnarray*}
\mathrm{C}(d_K)&\longrightarrow&\mathrm{Gal}(H/K)\\
Q&\mapsto&(h(\theta)\mapsto h^{\beta_Q}(\theta_Q))|_H,
\end{eqnarray*}
where $h\in\mathcal{F}_N$ is defined and finite at $\theta$.
\end{proposition}
\begin{proof}
See \cite{Gee} or \cite{Stevenhagen}.
\end{proof}

\begin{proposition}\label{Gal(K_{(N)}/K)}
Let $\min(\theta,\mathbb{Q})=X^2+BX+C\in\mathbb{Z}[X]$. For each
positive integer $N$, the matrix group
\begin{equation*}W_{N,\theta}=\bigg\{\begin{pmatrix}t-B s &
-C
s\\s&t\end{pmatrix}\in\mathrm{GL}_2(\mathbb{Z}/N\mathbb{Z})~:~t,s\in\mathbb{Z}/N\mathbb{Z}\bigg\}
\end{equation*}
gives rise to a surjection
\begin{eqnarray*}
W_{N,\theta}&\longrightarrow&\mathrm{Gal}(K_{(N)}/H)\\
\alpha&\mapsto&(h(\theta)\mapsto h^\alpha(\theta)),
\end{eqnarray*}
where $h\in\mathcal{F}_N$ is defined and finite at $\theta$. The
action of $\alpha$ on $\mathcal{F}_N$ is the action as an element of
$\mathrm{GL}_2(\mathbb{Z}/N\mathbb{Z})/\{\pm1_2\}$
($\simeq\mathrm{Gal}(\mathcal{F}_N/\mathcal{F}_1)$). If
$K\neq\mathbb{Q}(\sqrt{-1})$, $\mathbb{Q}(\sqrt{-3})$, then the
kernel is $\{\pm1_2\}$.
\end{proposition}
\begin{proof}
See \cite{Gee} or \cite{Stevenhagen}.
\end{proof}

Combining the above two propositions we achieve the next result.

\begin{proposition}\label{conjugate}
Let $K$ be an imaginary quadratic field other than
$\mathbb{Q}(\sqrt{-1})$ and $\mathbb{Q}(\sqrt{-3})$, and $N$ be a
positive integer. Then we have a bijective map
\begin{eqnarray*}
\begin{array}{cccc}
W_{N,\theta}/\{\pm1_2\}\times\mathrm{C}(d_K)&\longrightarrow&
\mathrm{Gal}(K_{(N)}/K)&\\
\alpha\times Q&\longmapsto&(h(\theta)\mapsto h^{\alpha\cdot
\beta_Q}(\theta_Q)),\end{array}
\end{eqnarray*}
where $h\in\mathcal{F}_N$ is defined and finite at $\theta$.
\end{proposition}
\begin{proof}
See \cite{J-K-S} Theorem 3.4.
\end{proof}

Proposition \ref{conjugate} and the following transformation formula
of Siegel functions enable us to find all conjugates of the singular
value $g_{(0,1/N)}(\theta)^{-12N/\gcd(6,N)}$, which will be used to
prove our main theorem.

\begin{proposition}\label{F_N}
Let $N\geq 2$. For $(v,w)\in\mathbb{Z}^2-N\mathbb{Z}^2$, the
function $g_{(v/N,w/N)}(\tau)^{-12N/\gcd(6,N)}$ is determined by
$\pm(v/N,w/N)\pmod{\mathbb{Z}^2}$. It belongs to $\mathcal{F}_N$,
and $\alpha\in\mathrm{GL}_2(\mathbb{Z}/N\mathbb{Z})/\{\pm1_2\}$
($\simeq\mathrm{Gal}(\mathcal{F}_N/ \mathcal{F}_1)$) acts on the
function by
\begin{equation*}
(g_{(v/N,w/N)}(\tau)^{-12N/\gcd(6,N)})^\alpha
=g_{(v/N,w/N)\alpha}(\tau)^{-12N/\gcd(6,N)}.
\end{equation*}
\end{proposition}
\begin{proof}
See \cite{K-S} Proposition 2.4 and Theorem 2.5.
\end{proof}

\section{Normal bases of ray class fields}

The following lemma is a slight modification of a result in
\cite{J-K-S}.

\begin{lemma}\label{newlemma}
Let $d_K$ ($\leq-7$) be the discriminant of an imaginary quadratic
field and $\theta$ be as in (\ref{theta}). Let $Q=aX^2+bXY+cY^2$ be
a reduced quadratic form in $\mathrm{C}(d_K)$ and $\theta_Q$ be as
in (\ref{theta_Q}). For each integer $N$ ($\geq2$) we have the
inequalities:
\begin{itemize}
\item[(i)] If $a\geq2$, then
$|g_{(s/N,t/N)}(\theta_Q)^{-1}/ g_{(0,1/N)}(\theta)^{-1}|< 1$
 for all $(s,t)\in\mathbb{Z}^2-N\mathbb{Z}^2$.
\item[(ii)] If $a=1$, then
$|g_{(s/N,t/N)}(\theta_Q)^{-1}/ g_{(0,1/N)}(\theta)^{-1}|< 1$ for
all $s,t\in\mathbb{Z}$ with $s\not\equiv0\pmod{N}$.
\item[(iii)] If $a=1$, then
$|g_{(0,t/N)}(\theta_Q)^{-1}/g_{(0,1/N)}(\theta)^{-1}|<1$ for all
$t\in\mathbb{Z}$ with $t\not\equiv0,\pm1\pmod{N}$.
\end{itemize}
\end{lemma}
\begin{proof}
See \cite{J-K-S} Lemmas 4.2$\sim$4.3.
\end{proof}

Now we are ready to prove our main theorem concerning a normal basis
of $K_{(N)}$ over $K$.

\begin{theorem}\label{main}
Let $K$ be an imaginary quadratic field other than
$\mathbb{Q}(\sqrt{-1})$ and $\mathbb{Q}(\sqrt{-3})$, and $\theta$ be
as in (\ref{theta}). Let $N$ ($\geq2$) be an integer. Then, for a
suitably large integer $m$, the conjugates of the singular value
\begin{equation*}
g_{(0,1/N)}(\theta)^{-12Nm/\gcd(6,N)}
\end{equation*}
form a normal basis of $K_{(N)}$ over $K$.
\end{theorem}
\begin{proof}
For simplicity we set $x=g_{(0,1/N)}(\theta)^{-12N/\gcd(6,N)}$,
which belongs to $K_{(N)}$ by Proposition \ref{F_N} and the main
theorem of complex multiplication (\cite{Lang} or \cite{Shimura}).
It suffices to show by Theorem \ref{criterion} that
\begin{equation}\label{tobeshown}
|x^\gamma/x|<1\quad\textrm{for all
$\gamma\in\mathrm{Gal}(K_{(N)}/K),~\gamma\neq\textrm{id}$}.
\end{equation}
To this end we consider by Propositions \ref{conjugate} and
\ref{F_N} a conjugate $x^\gamma$ of $x$ which is of the form
\begin{equation*}
x^\gamma=(g_{(0,1/N)}(\tau)^{-12N/\gcd(6,N)})^{\alpha\cdot
\beta_Q}(\theta_Q)=g_{(s/N,t/N)\beta_Q}(\theta_Q)^{-12N/\gcd(6,N)}
\end{equation*}
for some
$\alpha=\pm\left(\begin{smallmatrix}t-Bs&-Cs\\s&t\end{smallmatrix}\right)\in
W_{N,\theta}/\{\pm1_2\}$, where $\min(\theta,\mathbb{Q})=X^2+BX+C$,
and $Q=aX^2+bXY+cY^2$ is a reduced quadratic form in
$\mathrm{C}(d_K)$. If $a\geq2$, then the inequality
(\ref{tobeshown}) holds by Lemma \ref{newlemma}(i). If $a=1$, then
we derive from the conditions (\ref{reduced}) and (\ref{disc}) for a
reduced quadratic form that
\begin{eqnarray*}
Q=\left\{\begin{array}{ll} X^2-(d_K/4)Y^2 &
\textrm{for}~d_K\equiv0\pmod{4}\vspace{0.1cm}\\
X^2+XY+((1-d_K)/4)Y^2 & \textrm{for}~
d_K\equiv1\pmod{4},\end{array}\right.
\end{eqnarray*}
which yields $\beta_Q\equiv1_2\pmod{N}$ by the definitions
(\ref{u1}) and (\ref{u2}). Thus
$x^\gamma=g_{(s/N,t/N)}(\theta_Q)^{-12N/\gcd(6,N)}$. Moreover, if
$(s,t)\not\equiv(0,\pm1)\pmod{N}$, then the inequality
(\ref{tobeshown}) is also true by Lemma \ref{newlemma}(ii) and
(iii). It is not necessary to consider the remaining case when $a=1$
and $(s,t)\equiv(0,\pm1)\pmod{N}$, because
$\gamma=\mathrm{id}\in\mathrm{Gal}(K_{(N)}/K)$ in this case.
Therefore (\ref{tobeshown}) holds true for all
$\gamma\in\mathrm{Gal}(K_{(N)}/K)$, $\gamma\neq\mathrm{id}$, as
desired. This completes the proof.
\end{proof}

\begin{remark}\label{S-Rinvariant}
Let $K$ be an imaginary quadratic field and $\mathfrak{f}$ be a
nontrivial integral ideal of $K$. We denote by
$\mathrm{Cl}(\mathfrak{f})$ the ray class group modulo
$\mathfrak{f}$ and write $C_0$ for its unit class. By definition the
ray class field $K_\mathfrak{f}$ modulo $\mathfrak{f}$ is a finite
abelian extension of $K$ whose Galois group is isomorphic to
$\mathrm{Cl}(\mathfrak{f})$ via the Artin map. For
$C\in\mathrm{Cl}(\mathfrak{f})$ we take an integral ideal
$\mathfrak{c}$ in $C$ so that
$\mathfrak{f}\mathfrak{c}^{-1}=[z_1,z_2]$ with
$z=z_1/z_2\in\mathfrak{H}$. We define the \textit{Siegel-Ramachandra
invariant} $g_\mathfrak{f}(C)$ by
\begin{equation*}
g_\mathfrak{f}(C)=g_{(a/N,b/N)}(z)^{12N},
\end{equation*}
where $N$ is the smallest positive integer in $\mathfrak{f}$ and
$a$, $b$ are integers such that $1=(a/N)z_1+(b/N)z_2$. This value
depends only on the class $C$ and belongs to $K_\mathfrak{f}$
(\cite{K-L}).
\par
Ramachandra showed in \cite{Ramachandra} that $K_\mathfrak{f}$ can
be generated over $K$ by certain elliptic unit, but his invariant
involves overly complicated product of high powers of singular
values of the Klein forms and singular values of the
$\Delta$-function for practical use. Thus, Lang proposed in his book
(\cite{Lang} p.292) to find a simpler one by utilizing
Siegel-Ramachandra invariants, and Schertz (\cite{Schertz}) showed
that $g_\mathfrak{f}(C_0)$ (so any $g_\mathfrak{f}(C)$) generates
$K_\mathfrak{f}$ over $K$ under some conditions on $\mathfrak{f}$.
He further conjectured that $g_\mathfrak{f}(C_0)$ is always a ray
class invariant.
\par
On the other hand, the modulus $\mathfrak{f}$ can be enlarged by
class field theory so that we may assume
$\mathfrak{f}=N\mathcal{O}_K$ for an integer $N$ ($\geq2$). By
definition we see that
$g_\mathfrak{f}(C_0)=g_{(0,1/N)}(\theta)^{12N}$. If
$K\neq\mathbb{Q}(\sqrt{-1}),~\mathbb{Q}(\sqrt{-3})$, then Theorem
\ref{main} indicates that $g_\mathfrak{f}(C_0)$ is indeed a
primitive generator of $K_\mathfrak{f}$ over $K$, which would be an
answer for the Lang-Schertz conjecture.
\end{remark}

\begin{example}
Let $K=\mathbb{Q}(\sqrt{-5})$ and $N=6$, then $d_K=-20$ and
$\theta=\sqrt{-5}$. We get
\begin{eqnarray*}
&&\mathrm{C}(d_K)=\{Q_1=X^2+5Y^2,Q_2=2X^2+2XY+3Y^2\}\\
&&\theta_{Q_1}=\sqrt{-5},~\theta_{Q_2}=(-1+\sqrt{-5})/2,~
\beta_{Q_1}=\left(\begin{smallmatrix}
1&0\\0&1\end{smallmatrix}\right),~
\beta_{Q_2}=\left(\begin{smallmatrix}1&5\\3&2\end{smallmatrix}\right)\\
&&W_{N,\theta}/\{\pm1_2\}=
\{\left(\begin{smallmatrix}1&0\\0&1\end{smallmatrix}\right),
\left(\begin{smallmatrix}0&1\\1&0\end{smallmatrix}\right),
\left(\begin{smallmatrix}2&3\\3&2\end{smallmatrix}\right),
\left(\begin{smallmatrix}3&2\\2&3\end{smallmatrix}\right)\}.
\end{eqnarray*}
By Propositions \ref{conjugate} and \ref{F_N}, the conjugates of
$x=g_{(0,1/6)}(\sqrt{-5})^{-12}$ are
\begin{eqnarray*}
\begin{array}{ll}
x_1=x, & x_2=g_{(1/6,0)}(\sqrt{-5})^{-12}\\
x_3=g_{(3/6,2/6)}(\sqrt{-5})^{-12}, &
x_4=g_{(2/6,3/6)}(\sqrt{-5})^{-12}\\
x_5=g_{(3/6,2/6)}((-1+\sqrt{-5})/2)^{-12}, &
x_6=g_{(1/6,5/6)}((-1+\sqrt{-5})/2)^{-12}\\
x_7=g_{(3/6,1/6)}((-1+\sqrt{-5})/2)^{-12}, &
x_8=g_{(5/6,4/6)}((-1+\sqrt{-5})/2)^{-12}
\end{array}
\end{eqnarray*}
possibly with multiplicity. By using a computer program (such as
MAPLE) one can find that
\begin{eqnarray*}
|x_i/x_1|<10^{-4}<1/\#\mathrm{Gal}(K_{(6)}/K)=1/8\quad\textrm{for}~i=2,\cdots,8.
\end{eqnarray*}
Therefore $\{x_1,\cdots,x_8\}$ becomes a normal basis of $K_{(6)}$
over $K$ by Theorem \ref{criterion}. Moreover, one can show that the
minimal polynomial of $x$ would be
\begin{eqnarray*}
(X-x_1)\cdots(X-x_8)&=&X^8-1263840X^7+42016796X^6+72894400X^5\\
&&+15056640X^4-4525280X^3+167196X^2-1280X+1
\end{eqnarray*}
with integer coefficients (\cite{K-S} $\S$3).
\end{example}

\bibliographystyle{amsplain}

\end{document}